\documentclass[11pt]{article}
\usepackage{amsmath,booktabs,amsthm,amssymb}

\newtheorem{definition}{Definition}
\newtheorem{example}{Example}

\newtheorem{theorem}{Theorem}
\newtheorem{lemma}{Lemma}
\newtheorem{proposition}{Proposition}
\newtheorem{corollary}{Corollary}

\begin{document}

\title{Explicit crystalline lattices in rigid cohomology}

\author{George M. Walker}

\maketitle 

\begin{abstract}
Motivated by applications in point counting algorithms using $p$-adic cohomology, we give an explicit description of integral lattices in rigid cohomology spaces that $p$-adically approximate logarithmic crystalline cohomology modules. These lattices are expressed in terms of the global sections of the twisted logarithmic de Rham complex. We prove the main theorem for smooth proper hypersurfaces with a smooth hyperplane section, then deduce the result for the quotient of such a pair in weighted projective space. We show how these results may be used to reduce the necessary $p$-adic precision with which one must work to compute zeta functions of varieties over finite fields.
\end{abstract}

\section{Introduction}
Rigid cohomology was introduced by Berthelot \cite{ber1} as a $p$-adic analogue of $\ell$-adic \'{e}tale cohomology. In addition to providing a method to attack the Weil conjectures in a $p$-adic fashion, rigid cohomology has over the last decade become increasingly useful for explicit computations. In 2001 Kedlaya \cite{kedhyp} exhibited an algorithm for computing the zeta functions of hyperelliptic curves, which has since been extensively generalised, see for instance \cite{AKR,gerk,laudfib,laudranks}. In all of these algorithms, a major complexity issue stems from the fact that one is computing with vector spaces over a $p$-adic field. Of course, one cannot work exactly and must perform computations with truncations of $p$-adic expansions, and the question of to how many digits one must work to guarantee a provably correct output is a pertinent one. Another problem is that during an algorithm one must perform several ``divisions by $p$'', and loosely speaking each one amounts to losing a digit of $p$-adic precision. One must be able to bound before starting the algorithm this precision loss in order to obtain a provably correct output. A sensible idea, therefore, is to use crystalline cohomology, an \it integral \normalfont cohomology theory closely related to rigid cohomology, where there are no negative powers of $p$. Unfortunately, a crucial step in point-counting algorithms is to compute an explicit lift of the Frobenius morphism on an algebraic variety, and it is not known how to do this directly on crystalline cohomology. One can in many cases, however, generate a lattice inside rigid cohomology that in some way ``comes from'' rigid cohomology. Lauder in \cite{laudranks} constructs these ``crystalline 
bases'' for use in the refined fibration method for computing zeta functions of a subclass of surfaces that admit a fibration into hyperelliptic curves. 

In this paper we generalise this construction in Corollary \ref{CrysBasisCor} to any smooth variety with a smooth hyperplane section.  For its use in applications we also treat the case of the quasi-smooth quotient of such a pair in weighted projective space in Theorem \ref{QuoCrysBas}, omitting the details of the proof since they mirror closely those of the smooth case. We also, in Theorem \ref{CharPolyLossThm}, quantify the statement that if we make use of these lattices, there is much less precision loss than one would expect in computing the characteristic polynomial of Frobenius, and hence the required precision for calculations is much lower throughout the algorithm.

\section{Smooth hypersurfaces}\label{smoothcrys}

Let ${\bf F}_q$ be a finite field of characteristic $p$. Denote by ${\cal V}$ the ring of Witt vectors of ${\bf F}_q$, and let $K$ be the field of fractions of ${\cal V}$. Let $\sigma$ (resp. $\sigma_K$) denote the $q$-power Frobenius automorphism on ${\cal V}$ (resp. $K$). Let $\bar{{\cal X}}$ be a smooth proper subscheme of ${\bf P}^{n+1}_{{\cal V}}$ for some $n$, let ${\cal D}$ be a smooth divisor on $\bar{{\cal X}}$, and let ${\cal U}=\bar{{\cal X}}\setminus {\cal D}$. Note that in particular $(\bar{{\cal X}},{\cal D})$ is a smooth ${\cal V}$-pair in the sense of \cite[Definition 2.2.1]{AKR}. Let $\bar{X}$, $D$ and $U$ be the special fibres of $\bar{{\cal X}}$, ${\cal D}$ and ${\cal U}$ respectively. We will be particularly interested in the case where $\bar{X}$ is a smooth hypersurface in ${\bf P}^{n+1}_{{\bf F}_q}$ with $D$ a smooth hyperplane section. For our purposes an $F$-crystal is a finitely generated ${\cal V}$-module $M$ with a $\sigma$-linear map $F:M\rightarrow M$ that becomes an injection after tensoring with $K$, and an $F$-isocrystal is a finite dimensional $K$-vector space $V$ with an injective $\sigma_K$-linear map $F:V\rightarrow V$. For ease of notation we write $F$ for the action of Frobenius on the $F$-crystals crystalline cohomology $H^i_{\mathop{\rm cris}}(\star)$ and log-crystalline cohomology $H^i_{\mathop{\rm cris}}((\star,\star))$, or on the $F$-isocrystals rigid cohomology $H^i_{\mathop{\rm rig}}(\star)$. We shall write $M(m)$ to denote the $F$-crystal or $F$-isocrystal $M$ with Frobenius action twisted by $q^{-m}$.

\subsection{The lattices $H(\bar{X}), H(\bar{X})_{\mathop{\rm prim}}, H(\bar{X},\log D)$}
It is not practical to work directly with the image of $H^n_{\mathop{\rm cris}}(\bar{X})\rightarrow H^n_{\mathop{\rm rig}}(\bar{X})$. The rigid cohomology of affine schemes is far easier from a computational point of view: it is isomorphic with Monsky-Washnitzer cohomology, defined using differential forms on an affine scheme, as used in Kedlaya's original algorithm \cite{kedhyp}. With this in mind, in this section we define certain ${\cal V}$-lattices contained in $K$-vector spaces, following \cite[Definition 2.2]{laudranks} but generalising to any dimension.
\begin{definition}
Since $(\bar{{\cal X}},{\cal D})$ is a smooth ${\cal V}$-pair, by \cite[Proposition 1.9]{berthfinitude} there is an $F$-equivariant isomorphism
\begin{equation}\label{CompTher}
H^n_{\mathop{\rm cris}}(\bar{X})\otimes_{{\cal V}} K \cong H^n_{\mathop{\rm rig}}(\bar{X}).
\end{equation}
With this in mind, we define the $F$-stable full rank ${\cal V}$-lattice $H(\bar{X})\subset H^n_{\mathop{\rm rig}}(\bar{X})$ as 
\[
H(\bar{X}):=\mathop{\rm im}\bigl(H^n_{\mathop{\rm cris}}(\bar{X})\rightarrow H^n_{\mathop{\rm rig}}(\bar{X})\bigr).
\]
For $Z$ a smooth proper hypersurface in ${\bf P}_{{\bf F}_q}^{N+1}$ for some $N$, we define the primitive middle crystalline cohomology of $Z$ as the twisted $(N+1)$-st cohomology of the smooth pair $({\bf P}_{{\bf F}_q}^{N+1},Z)$. That is,
\[
H_{\mathop{\rm cris},\mathop{\rm prim}}^N(Z):=H^{N+1}_{\mathop{\rm cris}}(({\bf P}_{{\bf F}_q}^{N+1},Z))(1).
\]
The \it primitive middle rigid cohomology of $Z$ \normalfont is
\[
H_{\mathop{\rm rig},\mathop{\rm prim}}^N(Z)=H_{\mathop{\rm cris},\mathop{\rm prim}}^N(Z)\otimes K.
\]
\end{definition}
The reason why we work with primitive cohomology is described in the following lemma, an integral version of \cite[\S 3.3]{AKR}.
\begin{lemma}\label{PrimProp}
If $N$ is odd then
\[
H_{\mathop{\rm cris},\mathop{\rm prim}}^N(Z)\cong H_{\mathop{\rm cris}}^N(Z).
\]
If $N$ is even then
\[
H_{\mathop{\rm cris},\mathop{\rm prim}}^N(Z)\cong H_{\mathop{\rm cris}}^N(Z)/R.
\]
where $R$ is a rank one $F$-crystal on which the Frobenius action is given by multiplication by $q^{N/2}$.
\end{lemma}
\begin{proof}
Use the exact sequence of \cite[Proposition 2.2.8, Proposition 2.4.1]{AKR}
\[
\dots\rightarrow H^i_{\mathop{\rm cris}}({\bf P}_{{\bf F}_q}^{N+1})\rightarrow H^i_{\mathop{\rm cris}}(({\bf P}_{{\bf F}_q}^{N+1},Z))\rightarrow H^{i-1}_{\mathop{\rm cris}}(Z)(-1)\rightarrow H^{i+1}_{\mathop{\rm cris}}({\bf P}_{{\bf F}_q}^{N+1})\rightarrow \dots,
\]
and standard results on the cohomology of proper varieties: see \cite[Lemma 5.1.3]{GMWThesis}. 
\end{proof}

 We now return to the pair $(\bar{X},D)$. By \cite[Proposition 2.2.8, Proposition 2.4.1]{AKR} there is an exact sequence
\begin{equation}\label{thetaseq}
\dots\rightarrow H^{n-2}_{\mathop{\rm cris}}(D)(-1)\stackrel{\Theta}{\rightarrow} H^n_{\mathop{\rm cris}}(\bar{X})\rightarrow H^n_{\mathop{\rm cris}}((\bar{X},D))\rightarrow \dots.
\end{equation}
\begin{proposition}
There is an isomorphism
\begin{equation}\label{cristoprim}
\frac{H^n_{\mathop{\rm cris}}(\bar{X})}{\mathop{\rm im} (\Theta)}\cong H^n_{\mathop{\rm cris},\mathop{\rm prim}}(\bar{X}).
\end{equation}
\end{proposition}

\begin{proof}
For $n$ odd the result is immediate from equation (\ref{thetaseq}). For $n$ even one can calculate $\Theta$ explicitly using the de Rham cohomology long exact sequence 
\[
\dots\rightarrow H^{n-1}_{\mathop{\rm dR}}((\bar{{\cal X}},{\cal D}))\rightarrow H^{n-2}_{\mathop{\rm dR}}({\cal D})\stackrel{\Theta}{\rightarrow} H^n_{\mathop{\rm dR}}(\bar{{\cal X}})\rightarrow H^n_{\mathop{\rm dR}}((\bar{{\cal X}},{\cal D}))\rightarrow\dots,
\]
see \cite[Proposition 5.1.4]{GMWThesis}. 
\end{proof}
\begin{definition} Since ${\cal D}$ is also smooth, by \cite[\S 2.4]{shiho} there is an isomorphism
\[
H^n_{\mathop{\rm cris}}((\bar{X},D)) \otimes K\cong H^n_{\mathop{\rm rig}}(U)
\]
and so we may define the $F$-stable full rank ${\cal V}$-lattice
\[
H(\bar{X},\log D):=\mathop{\rm im}\bigl(H^n_{\mathop{\rm cris}}((\bar{X},D))\rightarrow H^n_{\mathop{\rm rig}}(U)\bigr).
\] 
\end{definition}
\begin{definition}
Define the lattice
\[
H(\bar{X})_{\mathop{\rm prim}}:=\mathop{\rm im}\bigl(H^n_{\mathop{\rm cris},\mathop{\rm prim}}(\bar{X})\rightarrow H^n_{\mathop{\rm cris}}((\bar{X},D)) \rightarrow H^n_{\mathop{\rm rig}}(U)\bigr)
\]
where the second map is the injection modulo torsion by the work in \cite[\S 2.4]{shiho}, which holds since we are assuming $D$ to be smooth.
\end{definition}

Note that by Lemma \ref{PrimProp}, the following formula holds for the zeta function of $\bar{X}$
\begin{equation}\label{zetaformula}
Z(\bar{X},T)=\frac{\det\bigl(1-FT|H(\bar{X})_{\mathop{\rm prim}}\bigr)^{(-1)^{n+1}}}{(1-T)(1-qT)\dots(1-q^nT)}.
\end{equation}
 
Using the sequence (\ref{thetaseq}) and excision with projective space, we obtain the following sequence \cite[Equation (5.6)]{GMWThesis} that modulo torsion is exact,
\begin{equation}\label{MHStors}
0\rightarrow H^n_{\mathop{\rm cris},\mathop{\rm prim}}(\bar{X}) \rightarrow H^n_{\mathop{\rm cris}}((\bar{X},D))\rightarrow H^{n-1}_{\mathop{\rm cris},\mathop{\rm prim}}(D)(-1)\rightarrow 0.
\end{equation}
We shall write $H(D)_{\mathop{\rm prim}}(-1)$ for the image of the quotient $H^n_{\mathop{\rm cris}}((\bar{X},D))/H^n_{\mathop{\rm cris},\mathop{\rm prim}}(\bar{X})$ as a lattice in $H^n_{\mathop{\rm rig}}(U)$.
By (\ref{MHStors}) one can retrieve the action of Frobenius on $H(\bar{X})_{\mathop{\rm prim}}$ from the action of Frobenius on $H(\bar{X},\log D)$ provided we know how Frobenius acts on $H(D)_{\mathop{\rm prim}}$. In our applications we may apply this observation using an induction argument if $D$ is a smooth hyperplane section of a smooth hypersurface $\bar{X}$.

\subsection{An approximate lattice $H(\bar{X},kD)$}\label{HXkD}

 Our aim is to compute the lattice $H(\bar{X},\log D)\subset H^n_{\mathop{\rm rig}}(U)$. In fact we define yet another lattice, $H(\bar{X},kD)$, which is ``$p$-adically'' close to $H(\bar{X},\log D)$, this term being made more precise later. It is for $H(\bar{X},kD)$ that we provide an explicitly computable basis, and in applications it is on this lattice that we shall compute an approximation to Frobenius. The work in this section is similar to that regarding surfaces in \cite[\S 3]{laudranks}, but some new proofs are required for the case of arbitrary dimension.

\begin{proposition}\label{vanprop}
Suppose ${\cal F}^{\bullet}$ is a complex of coherent sheaves of ${\cal O}_{\bar{{\cal X}}}$-modules such that 
\[
\Omega^{\bullet}_{(\bar{{\cal X}},{\cal D})}\hookrightarrow {\cal F}^{\bullet}\otimes_{{\cal O}_{\bar{{\cal X}}}}{\cal O}_{\bar{{\cal X}}}(l{\cal D})
\]
induces maps on the homology sheaves whose kernels and cokernels are killed by $\mathop{\rm lcm}\{1,\dots,l+1\}$. Assume further that the map on zero-th homology sheaves is an isomorphism. Let $k$ be an integer such that the sheaves ${\cal F}^i\otimes_{{\cal O}_{\bar{{\cal X}}}}{\cal O}_{\bar{{\cal X}}}(k{\cal D})$ are acyclic for all $i\ge 0$ (such a $k$ exists by Serre's vanishing theorem \cite[Theorem 5.2]{harts}). Then there exists a map
\begin{equation}\label{approxlatticequo}
H^n_{\mathop{\rm cris}}((\bar{X},D))\rightarrow \frac{\Gamma\bigl(\bar{{\cal X}},{\cal F}^n\otimes_{{\cal O}_{\bar{{\cal X}}}}{\cal O}_{\bar{{\cal X}}}(k{\cal D})\bigr)}{d\Bigl(\Gamma\bigl(\bar{{\cal X}},{\cal F}^{n-1}\otimes_{{\cal O}_{\bar{{\cal X}}}}{\cal O}_{\bar{{\cal X}}}(k{\cal D})\bigr)\Bigr)},
\end{equation}
whose kernel and cokernel are killed by multiplication by $p^{n\lfloor \log_p (k+1) \rfloor}$.
\end{proposition}
\begin{proof}
First, we have the canonical isomorphism \cite[Theorem 6.4]{kato} with logarithmic de Rham cohomology,
\begin{equation}\label{lat1}
H^n_{\mathop{\rm cris}}((\bar{X},D))\cong H^n_{\mathop{\rm dR}}((\bar{{\cal X}},{\cal D})).
\end{equation}
By definition,
\begin{equation}\label{lat2}
H^n_{\mathop{\rm dR}}((\bar{{\cal X}},{\cal D}))={\bf R}^n\Gamma \Omega^{\bullet}_{(\bar{{\cal X}},{\cal D})},
\end{equation}
where $\Omega^{\bullet}_{(\bar{{\cal X}},{\cal D})}$ is the complex of sheaves of algebraic differentials on $\bar{{\cal X}}$ with logarithmic singularities along ${\cal D}$. By consideration of the spectral sequence \cite[Application 5.7.10]{weibel} from homology sheaves to hypercohomology
\[
E_2^{a,b}=H^a({\cal X},{\cal H}^b(\star))\Rightarrow {\bf R}^{a+b}\Gamma\star,
\]
we obtain maps between graded pieces of filtrations of ${\bf R}^n\Gamma \Omega^{\bullet}_{(\bar{{\cal X}},{\cal D})}$ and ${\bf R}^n\Gamma \bigl({\cal F}^{\bullet}\otimes_{{\cal O}_{\bar{{\cal X}}}}{\cal O}_{\bar{{\cal X}}}(k{\cal D})\bigr)$. By the hypotheses of the Proposition, $n$ of these maps have kernel and cokernel killed by multiplication by $\mathop{\rm lcm}\{1,\dots,k+1\}$, and one is an isomorphism. Thus there is a map
\begin{equation}\label{lat3}
{\bf R}^n\Gamma \Omega^{\bullet}_{(\bar{{\cal X}},{\cal D})}\rightarrow {\bf R}^n\Gamma \bigl({\cal F}^{\bullet}\otimes_{{\cal O}_{\bar{{\cal X}}}}{\cal O}_{\bar{{\cal X}}}(k{\cal D})\bigr),
\end{equation}
whose kernel and cokernel are killed by multiplication by $(\mathop{\rm lcm}\{1,\dots,k+1\})^n$. There is a spectral sequence \cite[Remark 2.1.6(ii)]{dimca} computing hypercohomology
\[
E_1^{a,b}=H^b\bigl({\cal X}, {\cal F}^{a}\otimes_{{\cal O}_{\bar{{\cal X}}}}{\cal O}_{\bar{{\cal X}}}(k{\cal D})\bigr)\Rightarrow {\bf R}^{a+b}\Gamma\bigl({\cal F}^{\bullet}\otimes_{{\cal O}_{\bar{{\cal X}}}}{\cal O}_{\bar{{\cal X}}}(k{\cal D})\bigr).
\]
Since ${\cal F}^i\otimes_{{\cal O}_{\bar{{\cal X}}}}{\cal O}_{\bar{{\cal X}}}(k{\cal D})$ are acyclic whenever $i\ge 0$, the only non-zero term on the $a+b=n$ diagonal of $E_1^{a,b}$ is 
\[
E_1^{n,0}=\Gamma\bigl(\bar{{\cal X}},{\cal F}^n\otimes_{{\cal O}_{\bar{{\cal X}}}}{\cal O}_{\bar{{\cal X}}}(k{\cal D})\bigr),
\]
and so 
\begin{equation}\label{lat4}
{\bf R}^n\Gamma \bigl({\cal F}^{\bullet}\otimes_{{\cal O}_{\bar{{\cal X}}}}{\cal O}_{\bar{{\cal X}}}(k{\cal D})\bigr)\cong E_{\infty}^{n,0}.
\end{equation}
Now since ${\cal F}^i\otimes_{{\cal O}_{\bar{{\cal X}}}}{\cal O}_{\bar{{\cal X}}}(k{\cal D})$ are acyclic whenever $i\ge 0$, all maps into and out of $E_r^{n,0}=$ for $r\ge 2$ are zero, and so
\begin{equation}\label{lat5}
E_{\infty}^{n,0}\cong E_2^{n,0} \cong \frac{\Gamma\bigl(\bar{{\cal X}},{\cal F}^n\otimes_{{\cal O}_{\bar{{\cal X}}}}{\cal O}_{\bar{{\cal X}}}(k{\cal D})\bigr)}{d\Bigl(\Gamma\bigl(\bar{{\cal X}},{\cal F}^{n-1}\otimes_{{\cal O}_{\bar{{\cal X}}}}{\cal O}_{\bar{{\cal X}}}(k{\cal D})\bigr)\Bigr)}.
\end{equation}
The result follows from the chain of maps (\ref{lat1}), (\ref{lat2}), (\ref{lat3}), (\ref{lat4}) and (\ref{lat5}).
\end{proof}
The following proposition gives us a very simple choice for ${\cal F}^{\bullet}$.
\begin{proposition}\label{kedgoodsheaves}
Suppose $(\bar{{\cal X}},{\cal D})$ is a smooth pair over ${\cal V}$. Then $\Omega^{\bullet}_{(\bar{{\cal X}},{\cal D})}$ satisfies the conditions for ${\cal F}^{\bullet}$ in Proposition \ref{vanprop}.
\end{proposition}
\begin{proof}
See \cite[Theorem 2.2.5]{AKR} for the claim on the cokernel and the isomorphism on zero-th cohomology sheaves. For the claim on the kernel, we use the notation of \cite[Theorem 2.2.5]{AKR}, and suppose that
\begin{equation}\label{horrid}
\Omega\in\ker\bigl({\cal H}^r(\Omega^{\bullet}_{(X,Z)/S})\rightarrow {\cal H}^r(\Omega^{\bullet}_{(X,S)/S}(l))\bigr).
\end{equation}
We show that $\Omega\equiv \tilde{\Omega}$ in ${\cal H}^r(\Omega^{\bullet}_{(X,Z)/S})$ for some $\tilde{\Omega}$ that is killed by multiplication by $\mathop{\rm lcm}\{1,\dots,l\}$.
By (\ref{horrid}), we have that $\Omega=d(\omega)$ for some $\omega\in Q_{D}\otimes \Omega^{r-1}_{(X,Z)/S}$. Write $\omega=\sum_U g_U \tilde{d}x_U$, where $U$ runs over $(r-1)$-element subsets of $\{1,\dots,n\}$. Let $h_U\in \mu^{-1}R_
T$ be the image of $g_U$ under the map $Q_D/Q_{D'}\rightarrow \mu^{-1}R_T$, so that $\Omega=d(\sum_U h_U \tilde{d}x_U)+d(\omega')$, where $\omega'\in Q_{D'}\otimes\Omega^{r-1}_{(X,Z)/S}$. Since $\gcd(j_i:i\in T)$ kills the homology of the complex $(Q_D/Q_{D'})\otimes \Omega^{\bullet}_{(X,Z)/S}$, we have that $d(\sum_U h_U \tilde{d}x_U)$ is killed by $\gcd(j_i:i\in T)$, and so by induction on the cardinality of $D$, 
\[
\Omega=\tilde{\Omega}+d(\tilde{\omega}),
\]
where $\tilde{\Omega}$ is killed by multiplication by $\mathop{\rm lcm}\{1,\dots,l\}$ and $\tilde{\omega}\in \Omega^{r-1}_{(X,Z)/S}$. So in particular, $d(\tilde{\omega})$ vanishes in ${\cal H}^r(\Omega^{\bullet}_{(X,Z)/S})$, and $\Omega\equiv\tilde{\Omega}$, as required.
\end{proof}

It seems hard to find a suitable $k$ of Proposition \ref{vanprop} in complete generality. We can, however, find a suitable $k$ for when ${\cal F}^{\bullet}$ is $\Omega_{(\bar{{\cal X}},{\cal D})}^{\bullet}$ in the case where $\bar{{\cal X}}$ is a smooth hypersurface in ${\bf P}:={\bf P}^{n+1}_{{\cal V}}$ and ${\cal D}$ is a smooth hyperplane section. The following proposition is a generalisation of \cite[Theorem 11.5.2]{katzsarnak}, and is a Bott-Deligne type argument similar to that in \cite[\S 11.6]{katzsarnak}.

\begin{proposition}\label{VanThm1}
Let $\bar{{\cal X}}$ be a smooth hypersurface of degree $d$ in ${\bf P}:={\bf P}^{n+1}$. Then 
\begin{itemize}
\item[(i)] $H^i(\bar{{\cal X}},{\cal O}_{\bar{{\cal X}}}(e))=0$ for $i>0$ and $e>d-(n+2)$, \\
\item[(ii)] $H^i(\bar{{\cal X}},\Omega^j_{\bar{{\cal X}}}(e))=0$ for $i,j>0$ and $e>\max\{jd,(j+1)d-(n+2)\}$.
\end{itemize}
\end{proposition}
\begin{proof}
We first consider $H^i(\bar{{\cal X}},{\cal O}_{\bar{{\cal X}}}(e))$. There exists an exact sequence of sheaves on ${\bf P}$
\begin{equation}\label{VanProof0}
0\rightarrow {\cal O}_{{\bf P}}(e-d)\rightarrow {\cal O}_{{\bf P}}(e) \rightarrow \pi_{\star}({\cal O}_{\bar{{\cal X}}}(e))\rightarrow 0,
\end{equation}
where the first map is induced by multiplication by the defining equation of $\bar{{\cal X}}$, and $\pi$ is the immersion $\bar{{\cal X}}\hookrightarrow{\bf P}$.
Taking cohomology we obtain a long exact sequence, part of which is
\begin{equation}\label{VanProof-1}
\dots\rightarrow H^i({\bf P},{\cal O}_{{\bf P}}(e))\rightarrow H^i\bigl({\bf P},\pi_{\star}({\cal O}_{\bar{{\cal X}}}(e))\bigr)\rightarrow H^{i+1}({\bf P},{\cal O}_{{\bf P}}(e-d))\rightarrow \dots.
\end{equation}
Since $\pi$ is affine, there are isomorphisms $H^i(\bar{{\cal X}},{\cal O}_{\bar{{\cal X}}}(e))\cong H^i({\bf P},\pi_{\star}({\cal O}_{\bar{{\cal X}}}(e)))$ by \cite[Exercise III.8.2]{harts}. By the cohomology of projective space \cite[Theorem III.5.1]{harts} and (\ref{VanProof-1}) we have $H^i(\bar{{\cal X}},{\cal O}_{\bar{{\cal X}}}(e))=0$ for $i>0$ and $e-d>-(n+2)$. This gives $(i)$.

For $(ii)$, we begin by taking $j$-th exterior powers in the sequence  \cite[Theorem II.8.13]{harts} of sheaves on ${\bf P}$,
\begin{equation}\label{ProjSeq1}
0\rightarrow\Omega^1_{{\bf P}}\rightarrow {\cal O}_{{\bf P}}(-1)^{\oplus n+2}\rightarrow {\cal O}_{{\bf P}}\rightarrow 0,
\end{equation}
to obtain
\begin{equation}\label{ProjProof2}
0\rightarrow \Omega_{{\bf P}}^j\rightarrow {\cal O}_{{\bf P}}(-j)^{\oplus \binom{n+2}{j}}\rightarrow \Omega_{{\bf P}}^{j-1}\rightarrow 0.
\end{equation}
Tensoring this sequence with ${\cal O}_{{\bf P}}(e)$ and restricting to $\bar{{\cal X}}$, we get the short exact sequence
\begin{equation}\label{ProjProof3}
0\rightarrow \Omega_{{\bf P}}^j(e)\vert_{\bar{{\cal X}}}\rightarrow {\cal O}_{\bar{{\cal X}}}(e-j)^{\oplus \binom{n+2}{j}}\rightarrow \Omega_{{\bf P}}^{j-1}(e)\vert_{\bar{{\cal X}}}\rightarrow 0.
\end{equation}
Taking cohomology we obtain the long exact sequence, part of which is
\begin{equation}\label{ProjProof4}
\begin{split}
\dots\rightarrow H^i(\bar{{\cal X}},&{\cal O}_{\bar{{\cal X}}}(e-j))^{\oplus \binom{n+2}{j}} \rightarrow H^i(\bar{{\cal X}},\Omega_{{\bf P}}^{j-1}(e)\vert_{\bar{{\cal X}}}) \\
 & \rightarrow H^{i+1}(\bar{{\cal X}},\Omega_{{\bf P}}^j(e)\vert_{\bar{{\cal X}}})\rightarrow H^{i+1}(\bar{{\cal X}},{\cal O}_{\bar{{\cal X}}}(e-j))^{\oplus \binom{n+2}{j}}\rightarrow\dots.
\end{split}
\end{equation}
By $(i)$, we  get isomorphisms 
\begin{equation}\label{ProfProof4}
H^i(\bar{{\cal X}},\Omega_{{\bf P}}^{j-1}(e)\vert_{\bar{{\cal X}}})\cong H^{i+1}(\bar{{\cal X}},\Omega_{{\bf P}}^j(e)\vert_{\bar{{\cal X}}}),
\end{equation}
when $i>0$ and $e-j>d-(n+2)$. By repeating this argument until we hit cohomology of forms of degree $n+2$, we see that 
\begin{equation}\label{ProjProof5}
H^i(\bar{{\cal X}},\Omega_{{\bf P}}^{j}(e)\vert_{\bar{{\cal X}}})=0\mbox{ when }i>0\mbox{ and }e>d.
\end{equation} 
Note that although it serves our purpose, this is not necessarily the tightest bound on $e$, as if $i>j$ we only need to carry out as many isomorphisms (\ref{ProjProof4}) as take us to cohomology in degree $n+1$.

Now we consider the $j$-th exterior power of the sequence \cite[Theorem II.8.17]{harts} of sheaves on $\bar{{\cal X}}$
\begin{equation}
0\rightarrow {\cal O}_{\bar{{\cal X}}}(-d)\rightarrow \Omega^1_{{\bf P}}\vert_{\bar{{\cal X}}}\rightarrow \Omega^1_{\bar{{\cal X}}}\rightarrow 0.
\end{equation}
Since ${\cal O}_{\bar{{\cal X}}}(-d)$ has rank one, we have
\begin{equation}\label{VanProof1}
0\rightarrow \Omega_{\bar{{\cal X}}}^{j-1}(-d)\rightarrow \Omega^j_{{\bf P}}\vert_{\bar{{\cal X}}}\rightarrow \Omega_{\bar{{\cal X}}}^j\rightarrow 0.
\end{equation}
Twisting by ${\cal O}_{\bar{{\cal X}}}(e)$ and taking cohomology we get a long exact sequence
\begin{equation}\label{VanProof2}
\begin{split}
\dots\rightarrow H^i(\bar{{\cal X}},&\Omega_{{\bf P}}^j(e)\vert_{\bar{{\cal X}}})\rightarrow H^i(\bar{{\cal X}},\Omega_{\bar{{\cal X}}}^j (e))\rightarrow \\
 &H^{i+1}(\bar{{\cal X}},\Omega_{\bar{{\cal X}}}^{j-1}(e-d))\rightarrow H^{i+1}(\bar{{\cal X}},\Omega_{{\bf P}}^j(e)\vert_{\bar{{\cal X}}})\rightarrow\dots.
 \end{split}
\end{equation}
Using equation (\ref{ProjProof5}) we deduce that there are isomorphisms
\[
H^i(\bar{{\cal X}},\Omega_{\bar{{\cal X}}}^j (e))\cong H^{i+1}(\bar{{\cal X}},\Omega_{\bar{{\cal X}}}^{j-1}(e-d)),
\]
whenever $i>0$ and $e>d$. Repeated use of this isomorphism tells us that
\[
H^i(\bar{{\cal X}},\Omega_{\bar{{\cal X}}}^j (e))\cong H^{i+j}(\bar{{\cal X}},{\cal O}_{\bar{{\cal X}}}(e-jd)),
\]
when $i>0$ and $e>jd$, and the latter vanishes when $i+j>0$ and $e-jd>d-(n+2)$ by $(i)$. This gives us $(ii)$. 
\end{proof}

We now apply Proposition \ref{VanThm1} to our case of interest.
\begin{theorem}\label{SmoothVanThm}
Let $\bar{{\cal X}}$ be a smooth hypersurface of degree $d$ in ${\bf P}={\bf P}^{n+1}_{{\cal V}}$, and ${\cal D}$ a smooth hyperplane section. Then
\begin{itemize}
\item[(i)] $H^i(\bar{{\cal X}},{\cal O}_{\bar{{\cal X}}}(e))=0$ for $i>0$ and $e>d-(n+2)$, \\
\item[(ii)] $H^{i}(\bar{{\cal X}},\Omega^j_{(\bar{{\cal X}},{\cal D})}(e))=0$ for $i>0$ and $e>\max\left\lbrace jd,(j+1)d-(n+2)\right\rbrace$.
\end{itemize}
\end{theorem}

\begin{proof}
$(i)$ is from Theorem \ref{VanThm1}. For $(ii)$ consider the exact sequence \cite[Proposition 2.2.8]{AKR} of sheaves on $\bar{{\cal X}}$
\[
0\rightarrow \Omega^j_{\bar{{\cal X}}}\rightarrow \Omega^j_{(\bar{{\cal X}},{\cal D})}\rightarrow \iota_{\star}\Omega^{j-1}_{{\cal D}}\rightarrow 0,
\]
where $\iota:{\cal D}\hookrightarrow \bar{{\cal X}}$.
Twist by ${\cal O}_{\bar{{\cal X}}}(e)$ and take cohomology to get the long exact sequence
\begin{equation}\label{pairvan1}
\dots\rightarrow H^i(\bar{{\cal X}},\Omega_{\bar{{\cal X}}}^j(e))\rightarrow H^i(\bar{{\cal X}}, \Omega^j_{(\bar{{\cal X}},{\cal D})}(e))\rightarrow H^i({\cal D},\Omega^{j-1}_{{\cal D}}(e))\rightarrow \dots.
\end{equation}
We apply Proposition \ref{VanThm1} to $\bar{{\cal X}}\subset{\bf P}_{{\cal V}}^{n+1}$ and ${\cal D}\subset{\bf P}_{{\cal V}}^{n}$, noting that the degree of ${\cal D}$ is equal to $d$.

Firstly let $i,j>0$. Then 
\[
H^{i}(\bar{{\cal X}}, \Omega^j_{\bar{{\cal X}}}(e))=0
\]
 for $e>\max\left\lbrace jd,(j+1)d-(n+2)\right\rbrace$, and 
\[
H^{i}({\cal D},\Omega^{j-1}_{{\cal D}}(e))=0
\]
 for $e>\max\left\lbrace (j-1)d,jd-(n+1)\right\rbrace.$ This gives $(ii)$. 
\end{proof}

\begin{corollary}\label{CrysBasisCor}
Let $\bar{{\cal X}}$ be a smooth hypersurface of degree $d$ in ${\bf P}={\bf P}^{n+1}_{{\cal V}}$, and ${\cal D}$ a smooth hyperplane section. Let $k$ be an integer such that
\[
k>\max\left\lbrace nd,(n+1)d-(n+2)\right\rbrace .
\]
Then there exists a map
\begin{equation}\label{approxlatticequoTHM}
H^n_{\mathop{\rm cris}}((\bar{X},D))\rightarrow \frac{\Gamma\bigl(\bar{{\cal X}},\Omega^n_{(\bar{{\cal X}},{\cal D})}(k)\bigr)}{d\Bigl(\Gamma\bigl(\bar{{\cal X}},\Omega^{n-1}_{(\bar{{\cal X}},{\cal D})}(k)\bigr)\Bigr)},
\end{equation}
whose kernel and cokernel are killed by multiplication by $p^{n\lfloor \log_p (k+1) \rfloor}$. Moreover, if we let $H(\bar{X},kD)$ denote the image of the right hand side of (\ref{approxlatticequoTHM}) in $H^n_{\mathop{\rm rig}}(U)$, where $U=X\setminus D$, then there is a natural embedding $H(\bar{X},\log D)\rightarrow H(\bar{X},kD)$ with cokernel killed by multiplication by $p^{n\lfloor \log_p (k+1) \rfloor}$.
\end{corollary}
\begin{proof}
After applications of Proposition \ref{kedgoodsheaves}, Proposition \ref{vanprop} and Theorem \ref{SmoothVanThm}, we establish the existence of the map (\ref{approxlatticequoTHM}) whose cokernel is killed by multiplication by $p^{n\lfloor \log_p (k+1) \rfloor}$. It remains to show that $H(\bar{X},\log D)\rightarrow H(\bar{X},kD)$ is an embedding with cokernel killed by multiplication by $p^{n\lfloor \log_p (k+1) \rfloor}$. The kernel of (\ref{approxlatticequoTHM}) is torsion, so the kernel of $H(\bar{X},\log D)\rightarrow H(\bar{X},kD)$ must be torsion. However, $H(\bar{X},\log D)$ is torsion free by definition, so the kernel vanishes. The claim on the cokernel follows from the fact that the cokernel of (\ref{approxlatticequoTHM}) is killed by multiplication by $p^{n\lfloor \log_p (k+1) \rfloor}$.
\end{proof}

\section{Quasi-smooth hypersurfaces}\label{nonsmoothcrys}\label{quob}
Suppose there exists a smooth pair $(\tilde{X},\tilde{D})$ with $\tilde{X}$ a hypersurface of degree $\tilde{d}$ in ${\bf P}^{n+1}_{{\bf F}_q}$, and $\tilde{D}$ a smooth hyperplane section. Suppose there also exists a finite group 
\[
G\cong {\bf Z}/(a_0)\times {\bf Z}/(a_1) \times \dots \times {\bf Z}/(a_{n+1}),
\]
acting on $\tilde{X}$ such that $\bar{X}$ and $D$, the $G$-invariant parts in weighted projective space ${\bf P}^{n+1}_{{\bf F}_q}[a_0,\dots,a_{n+1}]$ of $\tilde{X}$ and $\tilde{D}$ respectively, are quasi-smooth. We may without losing generality assume that $\gcd(a_0,\dots,\hat{a_i},\dots,a_{n+1})=1$ for all $i$, where the hat denotes omission. Then we may define the lattices as the $G$-invariant parts of the lattices obtained from the smooth objects, following \cite[\S 4]{laudranks}.

\begin{theorem}\label{GinvLattices} The group $G$ acts on the lattices $H(\tilde{X})$, $H(\tilde{X})_{\mathop{\rm prim}}$ and $H(\tilde{X},\log \tilde{D})$. Define
\begin{itemize} 
\item $H(\bar{X}):=H(\tilde{X})^G$ \\
\item $H(\bar{X})_{\mathop{\rm prim}}:=H(\tilde{X})_{\mathop{\rm prim}}^G$ \\
\item $H(\bar{X},\log D):=H(\tilde{X},\log \tilde{D})^G$.
\end{itemize}
\end{theorem}
\begin{proof}
See the results in \cite[\S 4]{laudranks}. Although they are proved for $n=2$, the proofs for the case of general $n$ are identical.
\end{proof}
Applying the analysis of \S \ref{smoothcrys} to the pair $(\tilde{X},\tilde{D})$ one obtains the following theorem.
\begin{theorem}\label{QuoCrysBas}
Let $(\tilde{{\cal X}},\tilde{{\cal D}})$ be a smooth pair with $\tilde{{\cal X}}$ a hypersurface of degree $\tilde{d}$ in ${\bf P}^{n+1}_{{\cal V}}$, and $\tilde{{\cal D}}$ a smooth hyperplane section. Let $\tilde{X}$ and $\tilde{D}$ be the special fibres of $\tilde{{\cal X}}$ and $\tilde{{\cal D}}$ respectively. Let $G\cong {\bf Z}/(a_0)\times {\bf Z}/(a_1) \times \dots \times {\bf Z}/(a_{n+1})$ be a finite group acting on $\tilde{X}$, and suppose that $\bar{X}$ and $D$ are the $G$-invariant parts of $\tilde{X}$ and $\tilde{D}$ respectively. Let $H(\bar{X},\log D)$ be the lattice defined in Definition \ref{GinvLattices}. Furthermore, let $k$ be an integer such that
\[
k>\max\bigl\lbrace n\tilde{d},(n+1)\tilde{d}-(n+2)\bigr\rbrace .
\]
Then \begin{itemize}
\item[(i)] $H(\bar{X},\log D)$ is a full rank $F$-invariant ${\cal V}$-lattice in $H^n_{\mathop{\rm rig}}(U)$, \\
\item[(ii)] there exists a map
\begin{equation}\label{GinvapproxlatticequoTHM}
H^n_{\mathop{\rm cris}}((\tilde{X},\tilde{D}))^G\rightarrow \frac{\Gamma\bigl(\tilde{{\cal X}},\Omega^n_{(\tilde{{\cal X}},\tilde{{\cal D}})}(k)\bigr)^G}{d\Bigl(\Gamma\bigl(\tilde{{\cal X}},\Omega^{n-1}_{(\tilde{{\cal X}},\tilde{{\cal D}})}(k)\bigr)\Bigr)^G}=\frac{\Gamma\bigl(\bar{{\cal X}},\Omega^n_{(\bar{{\cal X}},{\cal D})}(k)\bigr)}{d\Bigl(\Gamma\bigl(\bar{{\cal X}},\Omega^{n-1}_{(\bar{{\cal X}},{\cal D})}(k)\bigr)\Bigr)}
\end{equation}
whose kernel and cokernel are killed by multiplication by $p^{n\lfloor \log_p (k+1) \rfloor}$. \\
\item[(iii)] If we let $H(\bar{X},kD)$ denote the image of the right hand side of (\ref{GinvapproxlatticequoTHM}) in $H^n_{\mathop{\rm rig}}(U)$, then there is a natural embedding $H(\bar{X},\log D)\rightarrow H(\bar{X},kD)$ with cokernel killed by multiplication by $p^{n\lfloor \log_p (k+1) \rfloor}$.
\end{itemize}
\end{theorem}
\begin{proof}
As in \cite[\S 4]{laudranks}, follow the arguments of \S \ref{smoothcrys} with $G$-invariant pieces. 
\end{proof}

\section{Application to point-counting algorithms}

\subsection{Smooth hypersurfaces}\label{smhypapp}
 For a background on how one may use $p$-adic cohomology theories to compute zeta functions of varieties over finite fields see \cite[Chapters 1 and 2]{GMWThesis}. Suppose $\bar{X}$ is a smooth proper hypersurface in ${\bf P}^{n+1}_{{\bf F}_q}$ and $D$ a smooth hyperplane section. We treat the case in which $(\bar{X},D)$ is the quotient of such a pair in weighted projective space in \S \ref{nonsmoothapp}. By (\ref{zetaformula}), in order to compute the zeta function of $\bar{X}$ we seek the polynomial $P_1(T):=\det(1-FT|H(\bar{X})_{\mathop{\rm prim}})$. By (\ref{MHStors}), we have an exact sequence of lattices
\begin{equation}\label{MHS}
0\rightarrow H(\bar{X})_{\mathop{\rm prim}}\rightarrow H(\bar{X}, \log D))\rightarrow H(D)_{\mathop{\rm prim}}(-1)\rightarrow 0
\end{equation}
and so by assuming we know the Frobenius action on $H(D)_{\mathop{\rm prim}}(-1)$ it is sufficient to compute the polynomial $P(T):=\det(1-FT|H(\bar{X}, \log D))$. Suppose we have computed integers $N_{i}$ such that we can recover $P(T)$ exactly given the $i$-th coefficient of $P(T)$ modulo $p^{N_{i}}$. These integers may be calculated using the constraints on the roots of $P(T)$ as dictated by the Weil conjectures in conjunction with the Newton-Girard identities, see \cite[Lemma 1.2.3]{AKR} and \cite[Lemma 6.1.1]{GMWThesis}. 
 Suppose we now have at our disposal a $p$-adic point-counting algorithm that computes a $p$-adic approximation $\tilde{F}$ to a matrix for $F$ acting on the lattice $H(\bar{X}, kD)$, for example, the direct method \cite{AKR}, the deformation method \cite{gerk}, or the fibration method \cite{laudfib}.
  We must now know to what accuracy the matrix $\tilde{F}$ must be computed to ensure that the $i$-th coefficient of $\det(1-\tilde{F}T)$ is equal to the $i$-th coefficient of $P(T)$ modulo $p^{N_{i}}$. We restrict to the case of $q=p$, so ${\cal V}={\bf Z}_p$ and $K={\bf Q}_p$; we remark at the end of the section on the case pf $p\ne q$. We make use of the compatibility of the Hodge structure on crystalline cohomology with the Frobenius map.

\begin{proposition}\label{FrobHodge}
There exist filtrations
\[
0=H(\bar{X})_{n+1}\subseteq H(\bar{X})_{n}\subseteq \cdots \subseteq H(\bar{X})_{0}=H(\bar{X})_{\mathop{\rm prim}}
\]
\[
0=H(D)_{n}\subseteq H(D)_{n-1}\subseteq \cdots \subseteq H(D)_{0}=H(D)_{\mathop{\rm prim}}
\]
such that the $p$-th power Frobenius $F$ acts with $F(H(\bar{X})_{i})\subseteq p^iH(\bar{X})_{\mathop{\rm prim}}$ and $F(H(D)_{i})\subseteq p^iH(D)_{\mathop{\rm prim}}$. We call these the (primitive) Frobenius-Hodge structures on $\bar{X}$ and $D$.

If we write
\[
h_{\bar{X},\mathop{\rm prim}}^{i,n-i}={\mathop{\rm rank}}(H(\bar{X})_{i}/H(\bar{X})_{i+1})
\]
and 
\[
h_{D,\mathop{\rm prim}}^{i,n-i-1}={\mathop{\rm rank}}(H(D)_{i}/H(D)_{i+1}),
\]
then $h_{\bar{X},\mathop{\rm prim}}^{i,j}={\mathop{\rm rank}}(H^j({\bar{\cal X}},\Omega^i_{{\bar{\cal X}}}))$ and $h_{D,\mathop{\rm prim}}^{i,j}={\mathop{\rm rank}}(H^j({\cal D},\Omega^i_{\cal D}))$ for all $i\ne j$, and $h_{\bar{X},\mathop{\rm prim}}^{i,i}={\mathop{\rm rank}}(H^i({\bar{\cal X}},\Omega^i_{{\bar{\cal X}}}))-1$ and $h_{D,\mathop{\rm prim}}^{i,i}={\mathop{\rm rank}}(H^i({\cal D},\Omega^i_{\cal D}))-1$.
\end{proposition}
\begin{proof}
See \cite[\S 5]{mazur} for the non-primitive case, from which this proposition is easily deduced using Lemma \ref{PrimProp}. 
\end{proof}

\begin{definition}
Define the integers $h_{(\bar{X},D)}^{i,n-i}$ as 
\[
h_{(\bar{X},D)}^{i,n-i}=\left\lbrace
\begin{array}{ll}
h_{\bar{X},\mathop{\rm prim}}^{0,n} & \mbox{ if } i=0 \\
h_{\bar{X},\mathop{\rm prim}}^{i,n-i}+h_{D,\mathop{\rm prim}}^{i-1,n-i} & \mbox{ if } i\ne 0. \\
\end{array}
\right.
\]
Define $\Gamma_{(\bar{X},D)}$ to be the lower convex hull of the points on the plane with coordinates
\[
\left\lbrace \left(\sum_{i=0}^j h^{i,n-i}_{(\bar{X},D)},\sum_{i=0}^j ih^{i,n-i}_{(\bar{X},D)}\right) \Big\vert 0\le j \le n \right\rbrace.
\]
The polygon $\Gamma_{(\bar{X},D)}$ should be thought of as the Hodge polygon of the pair $(\bar{X},D)$, since the Newton polygon of Frobenius acting on $H(\bar{X}, \log D)$ is bounded below by $\Gamma_{(\bar{X},D)}$ via Proposition \ref{FrobHodge} the exact sequence (\ref{MHS}).
\end{definition}

We now state the main theorem of this chapter which quantifies the loss of precision in computing the characteristic polynomial of Frobenius if we use a basis for $H(\bar{X}, \log D)$.

\begin{theorem}\label{CharPolyLossThm}
Suppose $\bar{X}$ is a smooth proper hypersurface in ${\bf P}^{n+1}_{{\bf F}_p}$ and $D$ a smooth hyperplane section. Suppose further that we have computed a basis ${\cal B}_{\mathop{\rm cris}}$ for the ${\bf Z}_p$-lattice $H(\bar{X},kD)$ using Corollary \ref{CrysBasisCor}, so ${\cal B}_{\mathop{\rm cris}}$ is also a basis for the ${\bf Q}_p$-vector space $H(\bar{X},kD)\otimes_{{\bf Z}_p}{\bf Q}_p$. Let $\tilde{A}$ be an approximation to the matrix for the Frobenius map $F$ acting on $H(\bar{X},kD)\otimes_{{\bf Z}_p}{\bf Q}_p$ with respect to ${\cal B}_{\mathop{\rm cris}}$ that is correct modulo $p^{N}$ where $N\ge n+n\lfloor\log_p(k+1)\rfloor+1$. Write $\det(1-FT|H(\bar{X}, \log D))=\sum_{i=0}^{m}a_i T^i$ and $\det(1-\tilde{A}T)=\sum_{i=0}^{m}\tilde{a}_i T^i$. Then
\[
\mathop{\rm ord}_p(a_i-\tilde{a}_i)\ge N+\Gamma_{(\bar{X},D)}(i)-n\lfloor\log_p(k+1)\rfloor,
\]
where $\Gamma_{(\bar{X},D)}(i)$ means the height of the polygon $\Gamma_{(\bar{X},D)}$ at $i$. 
\end{theorem}

The proof of this theorem mirrors closely that of \cite[Theorem 5.1]{laudranks}. We first prove a lemma for when we have a basis for $H(\bar{X},\log D)$ that is \it adapted to the Hodge structure\normalfont. That is, a basis on which the Frobenius action has a matrix of the form
\begin{equation}\label{HodgeBigMat}
A=\left(
\begin{array}{ccccccccc}
p^nA_n & p^{n-1}A_{n-1} & \dots & pA_1 & A_0 & B_{n-1} & B_{n-2} & \dots & B_0 \\
\bf{[0]} & \bf{[0]} & \dots & \bf{[0]} & \bf{[0]} & p^nC_{n-1} & p^{n-1}C_{n-2} & \dots & pC_0 \\
\end{array}
\right),
\end{equation}
where $A_i,B_i,C_i$ are all matrices over ${\bf Z}_p$ of the following sizes:

\begin{center}
\begin{tabular}{@{}cc@{}} \toprule
 Matrix & Size  \\  \midrule
 $A_i$ &   $h_{\bar{X},\mathop{\rm prim}}^{n-i,i}\times h_{\bar{X},\mathop{\rm prim}}^n$  \\
  $B_i$ &   $h_{D,\mathop{\rm prim}}^{n-i-1,i}\times h_{\bar{X},\mathop{\rm prim}}^n$  \\
  $C_i$ &   $h_{D,\mathop{\rm prim}}^{n-i-1,i}\times h_{D,\mathop{\rm prim}}^n$  \\ \bottomrule
\end{tabular}
\end{center}
and the $\bf{[0]}$ entries in (\ref{HodgeBigMat}) are zero matrices of the appropriate size. Such a basis exists by the Frobenius-Hodge structures of Proposition \ref{FrobHodge} and the exact sequence (\ref{MHS}).

\begin{lemma}\label{CharPolyLossLemma}
Let ${\cal B}_{\mathop{\rm cris}}$ be a basis for the ${\bf Z}_p$-lattice $H(\bar{X},\log D)$ which is adapted to the Hodge structure. Let $A=(a_{i,j})$ be the matrix for the Frobenius map $F$ acting on $H(\bar{X},\log D)$ with respect to ${\cal B}_{\mathop{\rm cris}}$, and $\tilde{A}=(\tilde{a}_{i,j})$ be an approximation to $A$ that is correct modulo $p^{N}$ where $N\ge 2(n+1)$. Write $\det(1-AT)=\sum_{l=0}^{m}a_l T^l$ and $\det(1-\tilde{A}T)=\sum_{l=0}^{m}\tilde{a}_l T^l$. Then
\[
\mathop{\rm ord}_p(a_l-\tilde{a}_l)\ge N+\Gamma_{(\bar{X},D)}(l),
\]
where $\Gamma_{(\bar{X},D)}(l)$ means the height of the polygon $\Gamma_{(\bar{X},D)}$ at $l$. 
\end{lemma}
\begin{proof}
The matrix $V=(v_{i,j})$ below is such that $\mathop{\rm ord}_p(a_{i,j}),\mathop{\rm ord}_p(\tilde{a}_{i,j})\ge v_{i,j}$:
\begin{equation}\label{HodgeBigMatVal}
V:=\left(
\begin{array}{ccccccccc}
\bf{[n]} & \bf{[n-1]} & \dots & \bf{[1]} & \bf{[0]} & \bf{[0]} & \bf{[0]} & \dots & \bf{[0]} \\
\bf{[N]} & \bf{[N]} & \dots & \bf{[N]} & \bf{[N]} & \bf{[n]} & \bf{[n-1]} & \dots & \bf{[1]} \\
\end{array}
\right),
\end{equation}
where the blocks are of the same sizes as those in (\ref{HodgeBigMat}).

In the proof of \cite[Lemma 5.3]{laudranks}, Lauder shows that 
\begin{equation}\label{MinTrans}
\mathop{\rm ord}_p(a_l-\tilde{a}_l)\ge N+\min \left\lbrace \sum_{i=1}^l v_{u_i,u_{\tau(i)}} - \min_{1\le i \le l} \{v_{u_i,u_{\tau(i)}}\}\right\rbrace,
\end{equation}
where the first minimum is taken over all choices of indices $1\le u_1 <\dots <u_l\le m$ and permutations $\tau\in S_l$.
By \cite[Lemma 5.5]{laudranks}, if $s$ of the $v_{u_i,u_{\tau(i)}}$ are taken from the top-right $h^{n-1}_{D,\mathop{\rm prim}}\times h^n_{\bar{X},\mathop{\rm prim}}$ submatrix of $V$, then $s$ are also taken from the bottom-left $h^n_{\bar{X},\mathop{\rm prim}}\times h^{n-1}_{D,\mathop{\rm prim}}$ submatrix. 
Since $N\ge2( n+1)$ by assumption, the minimum of $\sum_{i=1}^l v_{u_i,u_{\tau(i)}}$ must be attained when \it none \normalfont of the  $v_{u_i,u_{\tau(i)}}$ are taken from the top-right or the bottom-left. 
They are therefore all taken from distinct rows and columns of the top-left $h^n_{\bar{X},\mathop{\rm prim}}\times h^n_{\bar{X},\mathop{\rm prim}}$ submatrix or the bottom-right $h^{n-1}_{D,\mathop{\rm prim}}\times h^{n-1}_{D,\mathop{\rm prim}}$ submatrix. 
By the definition of $\Gamma_{(\bar{X},D)}$, the minimum the sum $\sum_{i=1}^l v_{u_i,u_{\tau(i)}}$ can (and does) attain is $\Gamma_{(\bar{X},D)}(l)$, with $\min_{1\le i \le l}\{v_{u_i,u_{\tau(i)}}\}=0$. This completes the proof. 
\end{proof}

Theorem \ref{CharPolyLossThm} is deduced from Lemma \ref{CharPolyLossLemma} and the simple change of basis arguments of \cite[Lemmas 5.6 and 5.7]{laudranks}. The following corollary is immediate.

\begin{corollary}
Suppose we have computed integers $N_{i}$ such that we can recover $P(T)=\det(1-FT|H(\bar{X}, \log D))$ exactly given the $i$-th coefficient of $P(T)$ modulo $p^{N_{i}}$. Then it is sufficient to compute the matrix $\tilde{F}$ approximating Frobenius on the lattice $H(\bar{X},kD)$ such that it is correct modulo $p^{N_F}$ where
\[
N_F=\max_i\{N_{i}-\Gamma_{(\bar{X},D)}(i)\}+n\lfloor\log_p(k+1)\rfloor.
\]
\end{corollary} 

 If $p\ne q=p^a$, the author expects that bounds similar to those in Theorem \ref{CharPolyLossThm} should hold for the $p$-power Frobenius map $F_p$, but \it not \normalfont for the $q$-power Frobenius map $F_q$, since Proposition \ref{FrobHodge} is only true for $F_p$. In this situation one can recover the matrix $A^{a,\sigma}$ for $F_q$ by the $\sigma$-linearity of $F_p$,
\[
A^{a,\sigma}=A\cdot A^{\sigma}\cdot A^{\sigma^2}\cdot {\dots}\cdot A^{\sigma^{a-1}}.
\]

\subsection{Quasi-smooth hypersurfaces}\label{nonsmoothapp}

The same results as in \S\ref{smhypapp} hold for quasi-smooth hypersurfaces in weighted projective space with a quasi-smooth hyperplane section. Once we know the existence of Frobenius-Hodge structures, the proofs are identical. The Frobenius-Hodge structures on the quotient of a smooth pair by a finite group $G$ can be obtained by taking $G$-invariant parts of cohomology and reducing to the smooth case. This is done in detail in \cite[\S 4.2]{laudranks} for surfaces, but the general case is the same \it mutatis mutandis\normalfont.

\section{Examples}

In the following examples we compute a basis for the lattice $H(\bar{X},kD)$. In each case the differential forms with logarithmic poles along a hyperplane section coincide exactly with differential forms with simple poles along that section, as illustrated in \cite[Proposition 7.4.9]{GMWThesis}. We compute the quotient in Theorem \ref{CrysBasisCor} by restricting to subrings of ${\cal V}$, for example ${\bf Z}/(p)$ as in \cite[\S 4.3.2]{laudranks} and the second example that follows, or the localisation ${\bf Z}_{(p)}$ as in the first example that follows. 

\begin{example}(A plane curve of degree 7)
Let $C$ be the plane curve defined over ${\bf F}_{5}$ whose affine part is given by the equation
\[
x^7+x^6y+3x^5y^2+x^3y^3+x^2y^5+3x^2y+2xy^6+2x+3y^7+y^5+3y^4+y+3=0.
\]
Then applying Theorem \ref{CrysBasisCor} we may take $k=12$, and by using linear algebra over the localisation ${\bf Z}_{(5)}$ we obtain the basis
\[
\left\lbrace xdy, xydy, xy^2dy,\dots, xy^{10}dy, x^2dy, x^2ydy,\dots, x^2y^9dy, x^3dy,\dots,x^3y^8dy,x^4dy,\dots,x^4y^5dy \right\rbrace
\]
for $H(\bar{C},kD)$. As expected, this basis has size 36, which is twice the genus of $C$ (the primitive $h^1$ of the curve) plus the degree minus one (the primitive $h^0$ of the hyperplane section).
\end{example}

\begin{example}(An elliptic surface)
Let $S$ be the surface defined over ${\bf F}_{11}$ whose affine part is given by the equation
\[
y^2=x^3+(t^{12}+t^9+3t^2+t+1)x+t^{18}+2t^{17}+t^{15}+t^7+t^3+1.
\]
Then applying Theorem \ref{CrysBasisCor} we may take $k=53$, and by using linear algebra over ${\bf Z}/{(11)}$ we obtain the basis
\[
\left\lbrace \frac{dx\wedge dt}{y},\frac{t dx\wedge dt}{y},\frac{t^2 dx\wedge dt}{y}\dots, \frac{t^{22} dx\wedge dt}{y},\frac{x dx\wedge dt}{y},\frac{xt dx\wedge dt}{y},\dots,\frac{xt^{10} dx\wedge dt}{y}\right\rbrace
\]
for $H(\bar{S},kD)$. As expected, this basis has size 34, which is 32 (the primitive $h^2$ of the surface) plus 2 (the primitive $h^1$ of the hyperplane section).
\end{example}

\end{document}